\documentclass[a4paper,10pt,leqno,oneside]{amsart}
\usepackage{amsmath}
\usepackage{amsfonts}
\usepackage{enumitem}
\usepackage{amsthm}
\usepackage{bm}
\usepackage{mathtools}
\usepackage{amssymb}
\usepackage[top=20truemm, bottom=20truemm, left=17truemm, right=17truemm]{geometry}

\usepackage[usenames,dvipsnames]{color}
\usepackage{xcolor}
\makeatletter \@addtoreset{equation}{section} \makeatother
\newcommand{\eref}[1]{(\ref{#1})}
\newcommand{\tref}[1]{Theorem \ref{#1}}
\newcommand{\pref}[1]{Proposition \ref{#1}}
\newcommand{\lref}[1]{Lemma \ref{#1}}
\newcommand{\cref}[1]{Corollary \ref{#1}}

\newcommand{\sref}[1]{Section \ref{#1}}

\theoremstyle{plain} \newtheorem{theorem}{Theorem}[section] \newtheorem{lemma}{Lemma}[section] \newtheorem{proposition}{Proposition}[section] 
\theoremstyle{definition} \newtheorem{remark}{Remark}[section] 
\colorlet{RED}{red}
\title[$L^\infty$-error estimate for isoparametric FEM]{Optimal $L^\infty$-error estimate for isoparametric\\ finite element method in a smooth domain}
\author[Takahito Kashiwabara]{Takahito Kashiwabara}
\address{Graduate School of Mathematical Sciences, The University of Tokyo, 3-8-1 Komaba, Meguro, 153-8914 Tokyo, Japan}
\email{tkashiwa@ms.u-tokyo.ac.jp}
\newcommand{\RED}[1]{\textcolor{black}{#1}}

\begin{document}
\begin{abstract}
We consider the isoparametric finite element method (FEM) for the Poisson equation in a smooth domain with the homogeneous Dirichlet boundary condition.
Because the boundary is curved, standard triangulated meshes do not exactly fit it.
Thereby we need to introduce curved elements if better accuracy than linear FEM is desired, which necessitates the use of isoparametric FEMs.
We establish optimal rate of convergence $O(h^{k+1})$ in the $L^\infty$-norm for $k \ge 2$, by extending the approach of our previous work [Kashiwabara and Kemmochi, Numer.\ Math.\ \textbf{144}, 553--584 (2020)] developed for Neumann boundary conditions and $k = 1$.
\end{abstract}
\maketitle

\section{Introduction}
Let $\Omega \subset \mathbb R^N \, (N = 2, 3)$ be a bounded domain of $C^{k+1,1}$-class ($k \ge 1$), with $\Gamma := \partial\Omega$.
We consider the following elliptic problem with the Dirichlet boundary condition:
\begin{align*}
	-\Delta u = f \quad \text{in} \quad \Omega, \qquad u = 0 \quad \text{on} \quad \Gamma.
\end{align*}
Its weak form reads, with $(\cdot, \cdot)_D$ denoting the $L^2(D)$-inner product,
\begin{equation} \label{eq: conti prob}
	(\nabla u, \nabla v)_\Omega = (f, v)_\Omega \quad \forall v \in H^1_0(\Omega).
\end{equation}
Throughout this paper, we assume that the solution $u$ of \eqref{eq: conti prob} is sufficiently smooth, especially, $u \in W^{k+1, \RED{\infty}}(\Omega)$ (hence $f \in W^{k-1, \RED{\infty}}(\Omega)$).

Let $\Omega_h$ be an approximate domain of $\Omega$ and $\mathring V_h \subset H^1_0(\Omega_h)$ be an approximate function space constructed by the isoparametric finite element method (FEM) of order $k \ge 2$, the details of which are given in Section \ref{sec2}.
The parameter $h > 0$ denotes the mesh size associated with $\Omega_h$ and $\mathring V_h$ (see \eqref{eq: def of h} below for the precise definition).
Then the finite element solution $u_h \in \mathring V_h$ is obtained by solving
\begin{equation} \label{eq: apx prob}
	a_h(u_h, v_h) := (\nabla u_h, \nabla v_h)_{\Omega_h} = (\tilde f, v_h)_{\Omega_h} \quad \forall v_h \in \mathring V_h,
\end{equation}
where $\tilde f \in W^{k-1, \infty}(\mathbb R^N)$ is an arbitrary extension of $f \in W^{k-1, \infty}(\Omega)$ satisfying the stability condition $\|\tilde f\|_{W^{k-1, \infty}(\mathbb R^N)} \le C \|f\|_{W^{k-1, \infty}(\Omega)}$.
The unique existence of the approximate solution $u_h$ of \eqref{eq: apx prob} is ensured by the Lax--Milgram theorem.
With this setting, we state the main result of this paper as follows.
\begin{theorem} \label{main thm}
	Let $k \ge 2$ and $\tilde u \in W^{k+1, \infty}(\mathbb R^N)$ be an arbitrary extension of $u$ such that $\|\tilde u\|_{W^{k+1, \infty}(\mathbb R^N)} \le C \|u\|_{W^{k+1, \infty}(\Omega)}$.
	Then we have
	\begin{equation*}
		\|\tilde u - u_h\|_{L^\infty(\Omega_h)} \le Ch^{k+1} \|u\|_{W^{k+1, \infty}(\Omega)},
	\end{equation*}
	where the constants $C$ depend only on $N$, $\Omega$, and $k$.
\end{theorem}

This type of pointwise error estimate has been of importance in theoretical studies of FEM.
For the case $\Omega_h = \Omega$ (thus no domain perturbation), it is known, cf.\ \cite[Theorem 2.1]{Sch98}, that optimal rate of convergence in the $L^\infty$-norm is obtained if one employs quadratic or higher-order approximation (i.e., $k \ge 2$ in our notation) for $\mathring V_h$ and that an additional factor $|\log h|$ has to be included for the linear approximation ($k = 1$).
A similar result which takes into account domain perturbation was obtained in \cite{ScWa82} assuming $\Omega_h \subset \Omega$ and using the maximum principle.
For the isoparametric FEM with $N = 2$ and $k = 2$, a nearly-optimal estimate $O(h^{3-\epsilon})$ was shown by \cite{Wah78}, where effects of numerical integration were also considered.

More recently, an $O(h^2|\log h|)$-estimate was proved for $k = 1$ in \cite{GLX2023}, adopting techniques from \cite{KasKem2020a} developed for a Neumann boundary condition.
An optimal $O(h^k)$-estimate in the $W^{1,\infty}$-norm was presented by \cite{DLM2023} based on the lift approach developed by \cite{EllRan2013}.
Compared to these researches, our main result above shows an optimal estimate $O(h^{k+1})$ in the $L^\infty$-norm without the use of $\Omega_h \subset \Omega$ or the maximum principle, by a natural extension of our previous result \cite{KasKem2020a} to isoparametric FEMs of order $k \ge 2$.
We would like to point out that the choice of the boundary condition is significant because an auxiliary Green's function estimate is available for the Dirichlet boundary condition (cf.\ \cite[Remark B.2]{KasKem2020a}), which enables us to exclude a logarithmic factor in $h$ (see \eqref{eq: thanks to DBC} below).

This paper is organized as follows.
In Sections \ref{sec2} and \ref{sec3}, following \cite[Section 2.2]{Kas2024} we introduce notations related to approximated domains and the isoparametric FEM.
In \sref{sec4}, estimation of $\|\tilde u - u_h\|_{L^\infty(\Omega_h)}$ is reduced to that of $\|\nabla(\tilde g - g_h)\|_{L^1(\Omega_h)}$, where $g$ and $g_h$ denote  a regularized Green function and its finite element approximation, respectively.
Then $\|\nabla(\tilde g - g_h)\|_{L^1(\Omega_h)}$ is addressed by weighted $H^1$- and $L^2$-error estimates as we present in Sections \ref{sec5} and \ref{sec6}.
We remark that interpolation estimates adapted to the Dirichlet boundary condition in approximate domains are non-trivial, which is discussed in detail in Propositions \ref{prop: I - mathring Ih global} and \ref{prop: I - mathring Ih local}.

After the completion of this work, we noticed that an optimal $L^\infty$-error estimate of isoparametric FEM in curvilinear polyhedra has been presented in \cite{LQXY2024}.
However, our strategy to prove \tref{main thm} differs from that of \cite{LQXY2024}, especially in that we neither utilize a maximum principle nor an ``exact triangulation of $\Omega$'' \cite{EllRan2013, Len86} (this requires that the nodes of $\partial\Omega_h$ lie exactly on $\partial\Omega$; cf.\ \cite[Remark 2.2]{Kas2024}).

\section{Approximation by isoparametric FEM} \label{sec2}
\subsection{Assumptions on $\Omega$}
Since $\Omega$ is $C^{k+1, 1}$-smooth, there exists a system of local coordinates $\{(U_r, \bm y_r, \varphi_r)\}_{r=1}^M$ such that $\{U_r\}_{r=1}^M$ forms an open covering of $\Gamma$, 
$\bm y_r = \RED{{}^t(y_{r,1}, \dots, y_{r,N-1}, y_{r,N}}) = {}^t(\bm y_r', \RED{y_{r,N}})$ is a rotated coordinate of $\bm x$, and $\varphi_r: \Delta_r \to \mathbb R$ gives a graph representation $(\bm y_r', \varphi_r(\bm y_r'))$ of $\Gamma \cap U_r$, where $\Delta_r$ is an open cube in $\mathbb R^{N-1}$. 
In view of $C^{k+1, 1}(\Delta_r) = W^{k+2, \infty}(\Delta_r)$,
\begin{equation*}
	\|(\nabla')^{m} \varphi_r\|_{L^\infty(\Delta_r)} \le C \quad (m = 0, \dots, k+2, \; r = 1, \dots, M)
\end{equation*}
for some constant $C > 0$, where $\nabla'$ means the gradient with respect to $\bm y_r'$.

We also introduce a notion of tubular neighborhoods $\Gamma(\delta) := \{\bm x\in\mathbb R^N \mid \operatorname{dist}(\bm x, \Gamma) \le \delta\}$.
It is known that \RED{(see e.g.\ \cite[Section 14.6]{GiTr98})} there exists $\delta_0>0$, which depends on the $C^{1,1}$-regularity of $\Omega$, such that each $\bm x \in \Gamma(\delta_0)$ admits a unique representation $\bm x = \bar{\bm x} + t \bm n(\bar{\bm x})$ with $\bar{\bm x} \in \Gamma$ and $t \in [-\delta_0, \delta_0]$.
Namely, $\bm\pi : \bm x \mapsto \bar{\bm x}$ is the orthogonal projection to $\Gamma$ and $t$ represents the signed distance function to $\Gamma$.

\subsection{Construction of approximate domains} \label{subsec: approximate domains}
\RED{Similarly to \cite[Section 2]{Kas2024},} we make the following assumptions \ref{H1}--\ref{H8} on finite element partitions and approximate domains.
We start from introducing a quasi-uniform family of triangulations $\{\tilde{\mathcal T}_h\}_{h \downarrow 0}$ of \RED{closed} straight $N$-simplices.
\begin{enumerate}[label=(H\arabic*)]
	\item \label{H1}
	Every $T \in \mathcal{\tilde T}_h$ is affine-equivalent to a fixed reference \RED{closed} simplex $\hat T$ of $\mathbb R^N$, 
	via the affine isomorphism $\tilde{\bm F}_{T}(\hat{\bm x}) = B_{T} \hat{\bm x} + \bm b_{T}$.
	The set $\tilde{\mathcal T}_h$ is mutually disjoint, that is, the intersection of every two different elements is either empty or agrees with their common face of dimension $\le N - 1$.
	
	\item $\{\tilde{\mathcal T}_h\}_{h \downarrow 0}$ is shape-regular and quasi-uniform, that is,
	\begin{equation} \label{eq: def of h}
		h_T \le C \rho_T \quad\text{and}\quad h := \max_{T' \in \RED{\tilde{\mathcal T}_h}} h_{T'} \le C h_T \quad (\forall T \in \tilde{\mathcal T}_h),
	\end{equation}
	where $h_T$ and $\rho_T$ stand for the diameter of $T$ and that of the largest ball contained in $T$, respectively.
\end{enumerate}

Let $\hat\Sigma_k = \{\hat{\bm a}_i\}_{i=1}^{N_k}$ denote the nodes in $\hat T$ of the Lagrange $\mathbb P_k$-finite element (see e.g.\ \cite[Section 2.2]{Cia78}).
The nodal basis functions $\hat\phi_i \in \mathbb P_k(\hat T)$ are then defined by $\hat\phi_i(\hat{\bm a}_j) = \delta_{ij}$ (the Kronecker delta) for $i, j = 1, \dots, N_k$.

Next we fix $\tilde T \in \tilde{\mathcal T}_h$, consider a parametric map $\bm F \in [\mathbb P_k(\hat T)]^N$ of the form $\bm F(\hat{\bm x}) = \sum_{i=1}^{N_k} \bm a_i \hat\phi_i(\hat{\bm x})$, and set $T := \bm F(\hat T)$.
We assume that $T$ is a perturbation of the linear element $\tilde T$ in the following sense:
\begin{enumerate}[label=(H\arabic*)]
	\setcounter{enumi}{2}
	\item \label{H4} $|\bm a_i - \tilde{\bm F}_{\tilde T}(\hat{\bm a}_i)| \le Ch_T^2$ for all $i = 1, \dots, N_k$ and $\tilde T \in \tilde{\mathcal T}_h$.
\end{enumerate}
If $h \le 1$ is small enough, which we assume throughout this paper, then such $\bm F$ becomes diffeomorphic on $\hat T$ (see \cite[Theorem 3]{CiaRav1972}).
Henceforth we write $\bm F_T$ and $\tilde{\bm F}_T$ to imply the $\bm F$ and $\tilde{\bm F}_{\tilde T}$ above, respectively.
From \ref{H4} one has $\operatorname{dist}(T, \tilde T) \le \|\bm F_T - \tilde{\bm F}_T\|_{L^\infty(\hat T)} \le Ch_T^2$, so that $\operatorname{diam} T \le Ch_T$.

Now we denote by $\mathcal T_h$ the set of $T$ constructed above, which is mutually disjoint and is called a partition into $\mathbb P_k$-isoparametric finite elements.
In addition, we define $\Omega_h$ to be the interior of the union of $\mathcal T_h$ (namely, $\overline\Omega_h = \bigcup_{T \in \mathcal T_h} T$).
Let us assume that $\{\mathcal T_h\}_{h \downarrow 0}$ is regular of order $k$ in the sense of \cite[Definition 3.2]{Ber1989}, that is,
\begin{enumerate}[label=(H\arabic*)]
	\setcounter{enumi}{3}
	\item \label{H6}
	$\|\nabla_{\hat{\bm x}}^m \bm F_T\|_{L^\infty(\hat T)} \le C\|B_T\|^m \le Ch_T^m$ for $T \in \mathcal T_h$ and $m = 1, \dots, k+1$,
\end{enumerate}
where $C$ is independent of $h_T$, and $\|B_T\|$ is the matrix $2$-norm of $B_T$.

\begin{remark} \label{rem: after regularity of order k}
	(i) We have $C_1 h_T^N \le \operatorname{meas}_N(T) = |\operatorname{det} (\nabla_{\hat{\bm x}} \bm F_T)| \le C_2 h_T^N$ for $T \in \mathcal T_h$.
	
	(ii) Under \ref{H6}, we find from \cite[Theorem 4]{CiaRav1972} or \cite[Theorem 4.3.3]{Cia78} that
	\begin{equation} \label{eq: high order derivative of Finv}
		\|\nabla_{\bm x}^m \bm F_T^{-1}\|_{L^\infty(T)} \le C h_T^{-1} \quad (m = 1, \dots, k+1).
	\end{equation}
\end{remark}

Let us next introduce a boundary mesh $\mathcal S_h$.
Setting $\Gamma_h := \partial\Omega_h$, we define
\begin{equation*}
	\mathcal S_h = \{ \bm F_T(\hat S) \subset \Gamma_h \mid T \in \mathcal T_h, \text{ $\hat S \subset \partial\hat T$ is a closed $(N-1)$-face of $\hat T$} \}.
\end{equation*}
Then we have $\Gamma_h = \bigcup_{S \in \mathcal S_h} S$ (disjoint union).
Each boundary element $S \in \mathcal S_h$ admits a unique $T \in \mathcal T_h$ such that $S \subset \partial T$, which is denoted by $T_S$.
We let $\bm b_r : U_r \to \mathbb R^{N-1}; {}^t(\bm y_r', \RED{y_{r,N}}) \mapsto \bm y_r'$ denote the projection to the base set.
We assume that $\Omega$ is approximated by $\Omega_h$ in the following sense.
\begin{enumerate}[label=(H\arabic*)]
	\setcounter{enumi}{4}
	\item \label{asmp: omit r}
	$\Gamma_h$ is covered by $\{U_r\}_{r=1}^M$, and each portion $\Gamma_h \cap U_r$ is represented as a graph $(\bm y_r', \varphi_{rh}(\bm y_r'))$, where $\varphi_{rh}$ is a continuous function defined in $\overline{\Delta_r}$.
	Moreover, each $S \in \mathcal S_h$ is contained in some $U_r$.
	
	\item \label{H8}
	The restriction of $\varphi_{rh}$ to $\bm b_r(S)$ for each $S \in \mathcal S_h$ is a polynomial function of degree $\le k$.
	Moreover, $\varphi_{rh}$ approximates $\varphi_r$ as accurately as a general $\mathbb P_k$-interpolation does; namely, we assume that
	\begin{align*}
		\|\varphi_r - \varphi_{rh}\|_{L^\infty(\bm b_r(S))} &\le Ch^{k+1} =: \delta, \\
		\|(\nabla')^m (\varphi_r - \varphi_{rh})\|_{L^\infty(\bm b_r(S))} &\le Ch^{k+1-m} \qquad (m = 1, \dots, k + 1).
	\end{align*}
\end{enumerate}
These assumptions imply that the local coordinate system for $\Omega$ is compatible with $\{\Omega_h\}_{h \downarrow 0}$ and that $\Gamma_h$ is a piecewise $\mathbb P_k$-accurate approximation of $\Gamma$.
If $h$ is small enough, \RED{$\bm\pi|_{\Gamma_h}$ is well-defined and bijective}, cf.\ \cite[Proposition 8.1]{KOZ16}.

\subsection{Estimates on a boundary-skin layer}
For a subset $\Gamma_h'$ of $\Gamma_h$, let $\bm\pi(\Gamma_h', \delta) := \{\bar{\bm x} + t \bm n(\bar{\bm x}) \mid \bar{\bm x} \in \bm\pi(\Gamma_h'), \; |t| \le \delta \}$ be a tubular neighborhood with the base $\Gamma_h'$.
We recall the following boundary-skin estimates for $S \in \mathcal S_h$, $1\le p\le \infty$, and $v \in W^{1,p}(\Omega \cup \Gamma(\delta))$:
\begin{gather}
	\|v\|_{L^p(\bm\pi(S, \delta))} \le C(\delta^{1/p} \|v\|_{L^p(S)} + \delta \|\nabla v\|_{L^p(\bm\pi(S, \delta))}), \label{eq2: boundary-skin estimates} \\
	\|v - v\circ\bm\pi\|_{L^p(S)} \le C\delta^{1-1/p} \|\nabla v\|_{L^p(\bm\pi(S, \delta))}. \label{eq3: boundary-skin estimates}
\end{gather}
The proofs are given in \cite[Theorems 8.1--8.3]{KOZ16} for $k = 1$, which can be extended to $k \ge 2$.
A version of \eref{eq2: boundary-skin estimates} limited to $\Omega_h \setminus \Omega$ also holds (see \cite[Lemma A.1]{KasKem2020a}):
\begin{equation} \label{eq: RHS with Omegah minus Omega}
	\|v\|_{L^p(\bm\pi(S, \delta) \cap (\Omega_h \setminus \Omega))} \le C(\delta^{1/p} \|v\|_{L^p(S)} + \delta \|\nabla v\|_{L^p(\bm\pi(S, \delta) \cap (\Omega_h \setminus \Omega))}).
\end{equation}

\subsection{Extension from $\Omega$ to $\Omega \cup \Gamma(\delta)$}
We let $\tilde\Omega := \Omega\cup\Gamma(\delta) = \Omega_h\cup\Gamma(\delta)$ with $\delta = Ch^{k+1}$ given above.
Recall that $\tilde u$ and $\tilde f$ are stable extensions of $u \in W^{k+1, \infty}(\Omega)$ and $f \in W^{k-1,\infty}(\Omega)$, respectively.

We also need extensions whose behavior in $\Gamma(\delta)\setminus\Omega$ can be controled by that in $\Gamma(c\delta) \cap \Omega$ for some constant $c>0$.
For this we introduce $P : W^{m,p}(\Omega) \to W^{m,p}(\tilde\Omega) \, (m \in \{0,1,2,3\}, p \in [1,\infty])$ by reflection with respect to $\Gamma$ as follows.
For $\bm x \in \Omega\setminus\Gamma(\delta)$ we let $Pf(\bm x) = f(\bm x)$; for $\bm x = \bar{\bm x} + t \bm n(\bar{\bm x}) \in \Gamma(\delta)$ we define
\begin{equation*}
	Pf(\bm x) =
	\begin{cases}
		f(\bm x) & (-\delta \le t<0), \\
		6f(\bar{\bm x} - tn(\bar{\bm x})) - 8f(\bar{\bm x} - 2t n(\bar{\bm x})) + 3f(\bar{\bm x} - 3t n(\bar{\bm x})) & (0\le t \le \delta).
	\end{cases}
\end{equation*}
The following stability result may be proved similarly to \cite[Proposition 2.1]{KasKem2020a}.
\begin{proposition} \label{prop: stability of extension}
	The extension operator $P$ satisfies
	\begin{align*}
		\|Pf\|_{W^{m, p}(\Gamma(\delta))} &\le C\|f\|_{W^{m, p}(\Omega \cap \Gamma(3\delta))} \quad m \in \{0,1,2,3\}, \; p\in[1,\infty],
	\end{align*}
	where the constant $C > 0$ is independent of $\delta$ and $f$.
\end{proposition}

\section{Finite element approximation} \label{sec3}
We introduce the global nodes of $\mathcal T_h$ by
\begin{equation*}
	\mathcal N_h = \{ \bm F_T(\hat{\bm a}_i) \in \overline\Omega_h \mid T \in \mathcal T_h, \; i = 1, \dots, N_k \}.
\end{equation*}
The interior and boundary nodes are denoted by $\mathring{\mathcal N}_h = \mathcal N_h \cap \operatorname{int}\Omega_h$ and $\mathcal N_h^\partial = \mathcal N_h \cap \Gamma_h$, respectively.
We next define the global nodal basis functions $\phi_{\bm a} \, (\bm a \in \mathcal N_h)$ by
\begin{equation*}
	\phi_{\bm a}|_T = \begin{cases}
		0 & \text{ if } \bm a \notin T, \\
		\hat\phi_i \circ \bm F_T^{-1} & \text{ if $\bm a \in T$ and  $\bm a = \bm F_T(\hat{\bm a_i})$ with $\hat{\bm a_i} \in \hat\Sigma_k$},
	\end{cases}
	\quad (\forall T \in \mathcal T_h)
\end{equation*}
which become continuous in $\overline\Omega_h$ thanks to the assumption on $\hat\Sigma_k$.
Then $\phi_{\bm a}(\bm b) = 1$ if $\bm a = \bm b$ and $\phi_{\bm a}(\bm b) = 0$ otherwise, for $\bm a, \bm b \in \mathcal N_h$.
We now set the $\mathbb P_k$-isoparametric finite element space $V_h$ by
\begin{equation*}
	V_h = \operatorname{span}\{\phi_{\bm a}\}_{\bm a \in \mathcal N_h} = \{ v_h \in C(\overline\Omega_h) \mid \hat v_h := v_h\circ \bm F_T \in \mathbb P_k(\hat T) \; (\forall T \in \mathcal T_h) \},
\end{equation*}
and its subspace $\mathring V_h \subset H^1_0(\Omega_h)$ with the zero boundary value constraint by
\begin{equation*}
	\mathring V_h = \operatorname{span}\{\phi_{\bm a}\}_{\bm a \in \mathring{\mathcal N}_h} = \{ v_h \in V_h \mid v_h = 0 \;\text{ on }\; \Gamma_h \}.
\end{equation*}

By the chain rule $\nabla_{\bm x} = (\nabla_{\bm x}\bm F_T^{-1}) \nabla_{\hat{\bm x}}$ and by \eref{eq: high order derivative of Finv}, we have
\begin{equation*}
	\|\nabla_{\bm x}^m v\|_{L^p(T)} \le Ch_T^{N/p} \sum_{l = 1}^m h_T^{-l} \|\nabla_{\hat{\bm x}}^l (v \circ \bm F_T)\|_{L^p(\hat T)} \quad (v \in W^{m,p}(T))
\end{equation*}
for $T \in \mathcal T_h, m \ge 1$, and $p \in [1, \infty]$.
In particular, if $v = v_h \in V_h$ and $m \ge k+1$ (thus $\nabla_{\hat{\bm x}}^{k+1} \hat v_h \equiv 0$ in $\hat T$) it follows from $\nabla_{\hat{\bm x}} = (\nabla_{\hat{\bm x}} \bm F_T) \nabla_{\bm x}$ and \ref{H6} that (cf.\ \cite[p.\ 187]{Wah78})
\begin{equation} \label{eq: superapx}
\begin{aligned}
	\|\nabla_{\bm x}^m v_h\|_{L^p(T)} &\le Ch_T^{N/p} \sum_{l = 1}^k h_T^{-l} \|\nabla_{\hat{\bm x}}^l \hat v_h\|_{L^p(\hat T)} \\
		&\le Ch_T^{N/p} \sum_{l = 1}^k h_T^{-l} \cdot h_T^{-N/p} h_T^l \|v_h\|_{W^{k, p}(T)} \le C \|v_h\|_{W^{k, p}(T)}.
\end{aligned}
\end{equation}

Next, for $v \in C(\overline\Omega_h)$ the Lagrange interpolation $I_h v \in V_h$ is defined by $I_h v = \sum_{\bm a \in \mathcal N_h} v(\bm a) \phi_{\bm a}$.
Accordingly, we define the interpolation to $\mathring V_h$ by $\mathring I_h v = \sum_{\bm a \in \mathring{\mathcal N}_h} v(\bm a) \phi_{\bm a}$, neglecting  values at the boundary nodes $\mathcal N_h^\partial$.
The difference between $I_h$ and $\mathring I_h$ is estimated in the next lemma (see the appendix for the proof).
\begin{lemma} \label{lem: Ih - mathring Ih}
	Assume that $T \in \mathcal T_h$ and $S \in \mathcal S_h$ satisfy $\emptyset \neq T \cap \Gamma_h \subset S$. Then we have
	\begin{equation*}
		\|\nabla^m(v_h - \mathring I_h v_h)\|_{L^p(T)} \le Ch_T^{1/p-m} \|v_h\|_{L^p(S)} \quad \forall v_h \in V_h,
	\end{equation*}
	where $m = 0, 1, \dots$ and $1\le p\le \infty$.
\end{lemma}
We recall from \cite[Theorem 5]{CiaRav1972} standard interpolation error estimates for $I_h$:
\begin{lemma} \label{lem: I - Ih}
	Let $l, m \in \mathbb N$ satisfy $0 \le l \le m \le k + 1$ and $p \in [1, \infty]$.
	Assume the embedding $W^{m,p} \hookrightarrow \RED{C^0}$ holds for subsets in $\mathbb R^N$.
	Then we have
	\begin{equation*}
		\|v - I_h v\|_{W^{l,p}(T)} \le Ch^{m-l} \|v\|_{W^{m,p}(T)} \quad (T \in \mathcal T_h, \, v \in W^{m, p}(T)).
	\end{equation*}
\end{lemma}
Estimates for $\mathring I_h$ are, however, more involved because of domain perturbation ($u = 0$ on $\Gamma$ does not necessarily imply $\tilde u = 0$ on $\Gamma_h$).
We state it in the following form, whose proof is similar to that of Proposition \ref{prop: I - mathring Ih local} below (we only have to consider global $\Omega_h$ and set $v_2 = 0$ there) and thus omitted here.
\begin{proposition} \label{prop: I - mathring Ih global}
	Under the assumptions of Lemma \ref{lem: I - Ih}, let $m \ge 2$ and $v \in W^{m, p}(\tilde\Omega)$ satisfy $v = 0$ on $\Gamma$.
	Then we have
	\begin{align*}
		\Big( \sum_{T \in \mathcal T_h} \|\nabla^l(v - \mathring I_h v)\|_{L^p(T)}^p \Big)^{1/p} &\le C h^{m - l} \|v\|_{W^{m,p}(\Omega_h)} + C h^{k+1-l} \|\nabla^2 v\|_{L^p(\Gamma(\delta))},
	\end{align*}
	with the obvious modification for $p = \infty$.
\end{proposition}

\section{Reduction to $W^{1,1}$-analysis of a regularized Green function} \label{sec4}
Fixing arbitrary $K \in \mathcal T_h$ and $z \in K$, we try to bound the pointwise error $\tilde u(\bm z) - u_h(\bm z)$.
We construct a regularized delta function; the proof is given in the appendix.
\begin{proposition} \label{prop: regularized delta}
	For $K \in \mathcal T_h$ and $\bm z \in K$, there exists $\eta = \eta_{K, \bm z} \in C^\infty_0(K)$ such that $\operatorname{dist}(\operatorname{supp}\eta, \partial K) \ge Ch_K$, $\|\nabla^m \eta\|_{L^\infty(K)} \le Ch_K^{-N-m} \, (m=0, 1)$, and
	\begin{equation*}
		(v_h, \eta)_K = v_h(\bm z) \quad \text{for $v_h = \hat v_h \circ \bm F_K^{-1}$ with arbitrary $\hat v_h \in \mathbb P_k(\hat T)$},
	\end{equation*}
	where the constant $C$ is independent of $K$, $\bm z$, and $h_K$.
\end{proposition}

Next we introduce a ``dyadic decomposition'' of $\Omega_h$.
We set a sequence of scales:
\begin{equation*}
	d_0 = L h, \quad d_j = 2^j d_0 \quad \text{for } \; j = 1, \dots, J := \left\lceil \frac{\log( \operatorname{diam}\Omega_h/d_0 )}{\log2} \right\rceil,
\end{equation*}
where $L$ means the ratio of the ``initial stride'' $d_0$ to the ``minimum scale'' $h$.
As we see later, $L$ will be taken sufficiently large (but independently of $h$).
Then we define a subset $\Omega_{h, j}$ of $\Omega_h$---which has the scale $d_j$ in terms of the distance from $K$---by
\begin{align*}
	\Omega_{h0} &= \bigcup\{T \in \mathcal T_h \mid d(T, K) \le d_0\}, \\
	\Omega_{h, j} &= \bigcup\{T \in \mathcal T_h \mid d_{j-1} < d(T, K) \le d_j\} \; (j = 1, \dots, J),
\end{align*}
where $d(T, T') = \min\{ |\bm x - \bm x'| \mid \bm x \in T, \bm x' \in T' \}$ denotes a distance function between two elements $T, T' \in \mathcal T_h$.
They are compatible with a standard ball $B(\bm z; r) = \{\bm x \mid |\bm x - \bm z| \le r\}$ and annulus $A(\bm z; r, R) = \{\bm x \mid r \le |\bm x - \bm z|\le R\}$.
In fact, by triangle inequalities, combined with $d_J \ge \operatorname{diam}\Omega_h$ and $\operatorname{diam} T \le Ch$ for $T \in \mathcal T_h$, we obtain
\begin{align*}
	\Omega_h &= \bigcup_{j=0}^J \Omega_{h, j} \text{ (disjoint union)}, \quad \Omega_{h0} \subset \Omega_h \cap B(\bm z; 2d_0), \\
	\Omega_{h, j} &\subset \Omega_h \cap A_j^{(s)} \subset \Omega_{h,j-1} \cup \Omega_{h, j} \cup \Omega_{h, j+1} =: \Omega_{h, j}' \, (j \ge 1),
\end{align*}
where $A_j^{(s)} := A(\bm z; (1-\frac{s}{2}) d_{j-1}, (1+s) d_j)$ \RED{for all $s \in (0, 1)$}, provided that $L$ is sufficiently large.
We also remark that $\sum_{j=\ell_1}^{\ell_2} d_j^\alpha$ is bounded by $C d_{\ell_1}^\alpha$ if $\alpha < 0$, by $C |\log d_0|$ if $\alpha = 0$, and by $C d_{\ell_2}^\alpha$ if $\alpha > 0$, for $0 \le \ell_1 \le \ell_2 \le J$.

Now let us start the first part of the proof of Theorem \ref{main thm}.
For any $v_h \in \mathring V_h$ we use the regularized delta function $\eta$ constructed in Proposition \ref{prop: regularized delta} to get
\begin{equation*}
	(\tilde u - u_h)(\bm z) = (\tilde u - v_h)(\bm z) + (v_h - \tilde u, \eta)_{\Omega_h} + (\tilde u - u_h, \eta)_{\Omega_h}.
\end{equation*}
The first two terms on the right-hand side are bounded by $C\|\tilde u - v_h\|_{L^\infty(K)}$.
To address the last term we define a regularized Green function $g \in W^{3,\infty}(\Omega)$ by solving
\begin{equation*}
	-\Delta g = \eta \quad\text{in }\; \Omega, \qquad g = 0 \quad\text{on }\; \Gamma,
\end{equation*}
where $\eta$ is extended by zero outside $\operatorname{supp} \eta \subset \Omega$ (this inclusion holds if $h$ is small).
We also utilize its finite element approximation $g_h \in \mathring V_h$ obtained by solving
\begin{equation*}
	a_h(v_h, g_h) = (\nabla v_h, \nabla g_h)_{\Omega_h} = (v_h, \eta)_{\Omega_h} \quad \forall v_h \in \mathring V_h.
\end{equation*}

Then it follows from $\operatorname{supp}\eta \subset \Omega_h \cap \Omega$ and $u_h = v_h = g_h = 0$ on $\Gamma_h$ that
\begin{align*}
	&(\tilde u - u_h, \eta)_{\Omega_h} = (u - u_h, -\Delta g)_{\Omega_h \cap \Omega} = (\tilde u - u_h, -\Delta \tilde g)_{\Omega_h} - (\tilde u - u_h, -\Delta \tilde g)_{\Omega_h \setminus \Omega} \\
	= \; &a_h(\tilde u - u_h, \tilde g) - (\tilde u - u_h, \partial_{\bm n_h} \tilde g)_{\Gamma_h} + (\tilde u - u_h, \Delta \tilde g)_{\Omega_h \setminus \Omega} \\
	= \; &a_h(\tilde u - u_h, \tilde g - g_h) + (\nabla(\tilde u - u_h), \nabla g_h)_{\Omega_h} - (\tilde u, \partial_{\bm n_h} \tilde g)_{\Gamma_h} + (\tilde u - u_h, \Delta \tilde g)_{\Omega_h \setminus \Omega} \\
	= \; &a_h(\tilde u - v_h, \tilde g - g_h) + (-\Delta \tilde u - \tilde f, g_h)_{\Omega_h \setminus \Omega} - (\tilde u, \partial_{\bm n_h} \tilde g)_{\Gamma_h} + (\tilde u - v_h, \Delta \tilde g)_{\Omega_h \setminus \Omega},
\end{align*}
where $\tilde g := Pg$, $\partial_{\bm n_h}$ denotes the derivative in the direction of the outer unit normal $\bm n_h$ to $\Gamma_h$, and we have used $a_h(v_h - u_h, \tilde g - g_h) = (u_h - v_h, \Delta \tilde g)_{\Omega_h \setminus \Omega}$ in the last line.

We estimate each term on the right-hand side above \RED{as follows:}
\begin{gather}
	|a_h(\tilde u - v_h, \tilde g - g_h)| \le \|\nabla(\tilde u - v_h)\|_{L^\infty(\Omega_h)} \|\nabla(\tilde g - g_h)\|_{L^1(\Omega_h)}, \notag \\
	\begin{aligned}
		|(-\Delta \tilde u - \tilde f, g_h)_{\Omega_h \setminus \Omega}| &\le C\|u\|_{W^{2,\infty}(\Omega)} \delta \|\nabla g_h\|_{L^1(\Omega_h \setminus \Omega)} \\
			&\le C \delta \|u\|_{W^{2,\infty}(\Omega)} (\|\nabla(\tilde g - g_h)\|_{L^1(\Omega_h \setminus \Omega)} + \|\nabla \tilde g\|_{L^1(\Gamma(\delta))}) \\
			&\le C \delta \|u\|_{W^{2,\infty}(\Omega)} (\|\nabla(\tilde g - g_h)\|_{L^1(\Omega_h)} + \delta),
	\end{aligned} \notag \\
	\begin{aligned}
		(\tilde u, \partial_{\bm n_h} \tilde g)_{\Gamma_h} &\le \|\tilde u - u \circ \bm\pi\|_{L^\infty(\Gamma_h)} \|\nabla \tilde g\|_{L^1(\Gamma_h)}
			\le C \delta \|\nabla\tilde u\|_{L^\infty(\Gamma(\delta))},
	\end{aligned}  \notag \\
	(\tilde u - v_h, \Delta \tilde g)_{\Omega_h \setminus \Omega} \le \|\tilde u - v_h\|_{L^\infty(\Omega_h)} \|\nabla^2 \tilde g\|_{L^1(\Gamma(\delta))} \le Ch^{-1} \delta \|\tilde u - v_h\|_{L^\infty(\Omega_h)}, \notag
\end{gather}
where we have exploited \eqref{eq2: boundary-skin estimates}--\eqref{eq: RHS with Omegah minus Omega} as well as the estimates
\begin{equation} \label{eq: thanks to DBC}
	\|\nabla \tilde g\|_{L^1(\Gamma(\delta))} \le C \delta, \quad \|\nabla \tilde g\|_{L^1(\Gamma_h)} \le C, \quad \|\nabla^2 \tilde g\|_{L^1(\Gamma(\delta))} \le C h^{-1} \delta.
\end{equation}

\begin{remark}
	The above results may be proved in the same way as \cite[Corollary B.1]{KasKem2020a}, which follows from the representation formula $g(\bm x) = \int_{\Omega \cap K} G(\bm x, \bm y) \eta(\bm y) \, d\bm y$ essentially, $G(\bm x, \bm y)$ denoting the Green function for the Laplace operator.
	We emphasize in particular that, for the present situation of the Dirichlet boundary condition, $G$ admits an auxiliary derivative estimates, which enables us to obtain the first two inequalities of \eqref{eq: thanks to DBC} (see \cite[Remark B.2]{KasKem2020a}).
\end{remark}

As a consequence, noting that $\delta \le h$ we have
\begin{align*}
	|\tilde u(\bm z) - u_h(\bm z)| &\le C \|\tilde u - v_h\|_{L^\infty(\Omega_h)} + \|\nabla(\tilde u - v_h)\|_{L^\infty(\Omega_h)} \|\nabla(\tilde g - g_h)\|_{L^1(\Omega_h)} \\
		&\qquad  + C \delta \|u\|_{W^{2,\infty}(\Omega)} (\|\nabla(\tilde g - g_h)\|_{L^1(\Omega_h)} + 1).
\end{align*}
Therefore, Theorem \ref{main thm} follows from $v_h = \mathring I_h \tilde u$ and interpolation estimates of Proposition \ref{prop: I - mathring Ih global}, provided that $\|\nabla(\tilde g - g_h)\|_{L^1(\Omega_h)} \le Ch$.
Because $\|\cdot\|_{L^1(\Omega_h)} = \sum_{j=0}^J \|\cdot\|_{L^1(\Omega_{h, j})}$ by the dyadic decomposition and $\|\cdot\|_{L^1(\Omega_{h, j})} \le C d_j^{N/2} \|\cdot\|_{L^2(\Omega_{h, j})}$, it suffices to show the weighted $H^1$-estimate
\begin{equation} \label{eq: weighted H1}
	\sum_{j=0}^J d_j^{N/2} \|\nabla(\tilde g - g_h)\|_{L^2(\Omega_{h, j})} \le Ch,
\end{equation}
which will be the goal of the next section.

\section{Weighted $H^1$-estimate} \label{sec5}
To prove \eqref{eq: weighted H1}, we first fix $j \in \{0, \dots, J\}$ and estimate $\|\nabla(\tilde g - g_h)\|_{L^2(\Omega_{h, j})}$.
Let $\omega \in C^\infty_0(\mathbb R^N)$ be a non-negative cut off function such that
\begin{equation*}
	\omega \equiv 1 \;\text{ in }\; A_j^{(1/3)}, \quad \operatorname{supp} \omega \subset A_j^{(2/3)}, \quad \|\nabla^m \omega\|_{L^\infty(\mathbb R^N)} \le C d_j^{-m} \; (m \ge 0).
\end{equation*}
From the construction of the dyadic decomposition, we see that $\omega$ is identically 1 and 0 in $\Omega_{h, j}$ and $\Omega \setminus \Omega_{h, j}'$, respectively.
Then it follows that
\begin{equation} \label{eq: start of local H1}
\begin{aligned}
	\|\nabla(\tilde g - g_h)\|_{L^2(\Omega_{h, j})}^2 &\le \int_{\Omega_h} \omega |\nabla(\tilde g - g_h)|^2 \, d\bm x \\
		&= a_h(\omega(\tilde g - g_h), \tilde g - g_h) - (\nabla\omega (\tilde g - g_h), \nabla(\tilde g - g_h))_{\Omega_h}.
\end{aligned}
\end{equation}
The second term \RED{of \eqref{eq: start of local H1}} is bounded by $Cd_j^{-1} \|\tilde g - g_h\|_{L^2(\Omega_{h, j}')} \|\nabla(\tilde g - g_h)\|_{L^2(\Omega_{h, j}')}$.

To deal with the first term on the right-hand side of \eqref{eq: start of local H1}, we set
\begin{equation*}
	e := \omega(\tilde g - g_h) = \omega(\tilde g - \mathring I_h \tilde g) + \omega(\mathring I_h \tilde g - g_h) =: e_1 + e_2.
\end{equation*}
Observe that, since $e_2 \in H^1_0(\Omega_h)$ and $a_h(\cdot, \tilde g - g_h) = (\cdot, -\Delta \tilde g)_{\Omega_h \setminus \Omega}$ in $\mathring V_h$,
\begin{equation} \label{eq: e1, e2, and other term}
	a_h(e, \tilde g - g_h) = a_h(e_1 - \mathring I_h e_1, \tilde g - g_h) + a_h(e_2 - I_h e_2, \tilde g - g_h) + (\mathring I_h e, - \Delta \tilde g)_{\Omega_h \setminus \Omega},
\end{equation}
and that
\begin{equation*}
	|a_h(e_1 - \mathring I_h e_1, \tilde g - g_h)| \le \|\nabla(e_1 - \mathring I_h e_1)\|_{L^2(\Omega_{h, j}')} \|\nabla(\tilde g - g_h)\|_{L^2(\Omega_{h, j}')}.
\end{equation*}
Here we need the following localized interpolation error estimate for $\mathring I_h$, which is announced before Proposition \ref{prop: I - mathring Ih global} and will be proved in the appendix.
\begin{proposition} \label{prop: I - mathring Ih local}
	Let $l, m \in \mathbb N$ satisfy $0 \le l \le m \le k + 1$, $m \ge 2$ and assume the embedding $W^{m,p} \hookrightarrow \RED{C^0}$ holds for subsets in $\mathbb R^N$.
	Moreover, assume that $v = v_1 + v_2$, where $v_1 \in W^{m,p}(\tilde\Omega)$ satisfies $v_1 = 0$ on $\Gamma$ and $v_2 \in C(\overline\Omega_h)$ satisfies $v_2|_T \in W^{m,p}(T)$ for $T \in \mathcal T_h$ with $v_2 = 0$ on $\Gamma_h$.
	Then, for $j =0, \dots, J$ we have
	\begin{equation} \label{eq: I - mathring Ih local}
		\Big( \sum_{T \subset \Omega_{h, j}} \|\nabla^l(v - \mathring I_h v)\|_{L^p(T)}^p \Big)^{1/p} \le C h^{m - l} \Big( \sum_{T \subset \Omega_{h, j}} \|v\|_{W^{m,p}(T)}^p \Big)^{1/p} + C h^{k+1-l} \|v_1\|_{W^{2,p}(\Omega_{h, j} \cup \bm\pi(\Gamma_{h, j}, \delta))},
	\end{equation}
	where $\Gamma_{h, j} := \Gamma_h \cap \partial\Omega_{h, j}$ and we make the obvious modification for $p = \infty$.
\end{proposition}

Setting $\Gamma_{h, j}' := \Gamma_h \cap \partial\Omega_{h, j}'$, we apply the above proposition twice ($v_1 = \omega \tilde g$ and $v_2 = \omega \mathring I_h \tilde g$ for the first time, and $v_1 = \tilde g$ and $v_2 = 0$ for the second time) to obtain
\begin{equation*}
\begin{aligned}
	\|\nabla(e_1 - \mathring I_h e_1)\|_{L^2(\Omega_{h, j}')} &\le Ch \Big( \sum_{T \subset \Omega_{h, j}'} \|\omega(\tilde g - \mathring I_h \RED{\tilde g})\|_{H^2(T)}^2 \Big)^{1/2} + Ch^k \|\omega \tilde g\|_{H^2(\Omega_{h, j}' \cup \bm\pi(\Gamma_{h, j}', \delta))} \\
		&\le Ch^2 \|\tilde g\|_{H^3(\Omega_{h, j}')} + Ch^k \|\tilde g\|_{H^2(\Omega_{h, j}' \cup \bm\pi(\Gamma_{h, j}', \delta))} + Ch^k \|\omega \tilde g\|_{H^2(\Omega_{h, j}' \cup \bm\pi(\Gamma_{h, j}', \delta))},
\end{aligned}
\end{equation*}
where the terms involving derivatives of $\omega$ are addressed by $hd_j^{-1} \le 1$.
If $j \le 1$, by an elliptic regularity estimate we have
\begin{equation*}
	\|\tilde g\|_{H^3(\Omega_{h, j}')} \le C \|g\|_{H^3(\Omega)} \le C \|\eta\|_{H^1(K)} \le C h^{N/2} h^{-N-1} = C L^{N/2+1} d_0^{-N/2-1}.
\end{equation*}
On the other hand, if $j \ge 2$ then we make use of derivative estimates for the Green function to deduce (cf.\ \cite[Lemma B.2]{KasKem2020a}, which shows $\|\tilde g\|_{H^2(\Omega_{h, j}')} \le C d_j^{-N/2}$)
\begin{equation*}
	\|\tilde g\|_{H^3(\Omega_{h, j}')} \le C d_j^{N/2} \|\tilde g\|_{W^{3,\infty}(\Omega_{h, j}')} \le Cd_j^{N/2} d_j^{-N-1} = C d_j^{-N/2-1}.
\end{equation*}
A similar argument bounds $\|\omega \tilde g\|_{H^2(\Omega_{h, j}' \cup \bm\pi(\Gamma_{h, j}', \delta))}$ (and also $\|\tilde g\|_{H^2(\Omega_{h, j}' \cup \bm\pi(\Gamma_{h, j}', \delta))}$) by
\begin{align*}
	&C d_j^{N/2} \Big( d_j^{-2} \|\tilde g\|_{L^\infty(\Omega_{h, j}' \setminus \Omega_{h, j} \cup \bm\pi(\Gamma_{h, j}', \delta))} + d_j^{-1} \|\nabla \tilde g\|_{L^\infty(\Omega_{h, j}' \setminus \Omega_{h, j} \cup \bm\pi(\Gamma_{h, j}', \delta))} \Big)
			+ C \|\tilde g\|_{H^2(\Omega_{h, j}' \cup \bm\pi(\Gamma_{h, j}', \delta))} \\
		\le \; &C d_j^{-N/2} |\log d_j| +
		\begin{cases}
			C L^{N/2} d_0^{-N/2} & (j \le 1), \\
			C d_j^{-N/2} & (j \ge 2).
		\end{cases}
\end{align*}
In view of $k \ge 2$, $d_j \le C$, and $L \ge 1$, we conclude that
\begin{equation} \label{eq: e1 - Ih e1}
	\|\nabla(e_1 - \mathring I_h e_1)\|_{L^2(\Omega_{h, j}')} \le C _j h^2 d_j^{-N/2-1},
\end{equation}
where $C_j = CL^{N/2+1}$ for $j \le 1$ and $C_j = C$ for $j \ge 2$.
Therefore, the first term of \eqref{eq: e1, e2, and other term} is estimated as
\begin{equation} \label{eq1: bound of (11)}
	|a_h(e_1 - I_h e_1, \tilde g - g_h)| \le C _j d_j^{-N/2-1} \|\nabla(\tilde g - g_h)\|_{L^2(\Omega_{h, j}')}.
\end{equation}

Next, to estimate the second term of \eqref{eq: e1, e2, and other term} we find from \lref{lem: I - Ih} that
\begin{align*}
	\|\nabla(e_2 - I_h e_2)\|_{L^2(\Omega_{h, j}')} &\le Ch^k \sum_{m=0}^{k+1} \Big( \sum_{T \subset \Omega_{h, j}'} \|\omega(\mathring I_h \tilde g - g_h)\|_{H^m(T)}^2 \Big)^{1/2} \\
		&\le Ch^k \sum_{m=0}^{k+1} d_j^{-m} \Big( \sum_{T \subset \Omega_{h, j}'} \|\mathring I_h \tilde g - g_h\|_{H^{k+1-m}(T)}^2 \Big)^{1/2}.
\end{align*}
The term for $m = 0$ may be omitted by \eqref{eq: superapx}, hence this is bounded by
\begin{align}
	&Ch^k \sum_{m=1}^{k+1} d_j^{-m} \Big( \sum_{T \subset \Omega_{h, j}'} \|\mathring I_h \tilde g - g_h\|_{H^{k+1-m}(T)}^2 \Big)^{1/2} \notag \\
	\le \; & Chd_j^{-1} \|\nabla(\mathring I_h \tilde g - g_h)\|_{L^2(\Omega_{h, j}')} + C d_j^{-1} \|\mathring I_h \tilde g - g_h\|_{L^2(\Omega_{h, j}')} \notag \\
	\le \; & Chd_j^{-1} \|\nabla(\tilde g - g_h)\|_{L^2(\Omega_{h, j}')} + C d_j^{-1} \|\tilde g - g_h\|_{L^2(\Omega_{h, j}')} + C_j h^2 d_j^{-N/2 - 1}, \label{eq: e2 - Ih e2}
\end{align}
where we have used $hd_j \le 1$, inverse inequalities $h^{l} \|\nabla^{l+1} v_h\|_{L^2(T)} \le C\|\nabla v_h\|_{L^2(T)}$, and an interpolation estimate for $\|\nabla(\tilde g - \mathring I_h \RED{\tilde g})\|_{L^2(\Omega_{h, j}')}$ based on \pref{prop: I - mathring Ih local} (which also appeared in dealing with $e_1$).
Consequently, the second term $a_h(e_2 - I_h e_2, \tilde g - g_h)$ is bounded by
\begin{equation} \label{eq2: bound of (11)}
\begin{aligned}
	&\|\nabla(e_2 - I_h e_2)\|_{L^2(\Omega_{h, j}')} \|\nabla(\tilde g - g_h)\|_{L^2(\Omega_{h, j}')} \\
	\le \; & Chd_j^{-1} \|\nabla(\tilde g - g_h)\|_{L^2(\Omega_{h, j}')}^2 + C d_j^{-1} \|\tilde g - g_h\|_{L^2(\Omega_{h, j}')} \|\nabla(\tilde g - g_h)\|_{L^2(\Omega_{h, j}')} \\
	&\qquad + C_j h^2 d_j^{-N/2 - 1} \|\nabla(\tilde g - g_h)\|_{L^2(\Omega_{h, j}')}.
\end{aligned}
\end{equation}

Finally, we apply \eqref{eq: RHS with Omegah minus Omega} (note that $\mathring I_h e = 0$ on $\Gamma_h$) to estimate the last term of \eqref{eq: e1, e2, and other term} as
\begin{align*}
	|(\mathring I_h e, - \Delta \tilde g)_{\Omega_h \setminus \Omega}| &\le C \delta \|\nabla (\mathring I_h e)\|_{L^2(\Omega_h \setminus \Omega)} \|\nabla^2 \tilde g\|_{L^2(\bm\pi(\Gamma_{h, j}', \delta))} \\
		&\le C \delta (\|\nabla e\|_{L^2(\Omega_h)} + \|\nabla(e - \mathring I_h e)\|_{L^2(\Omega_h)}) \cdot (\delta d_j^{N-1})^{1/2} d_j^{-N},
\end{align*}
where we refer to \cite[Lemma B.3]{KasKem2020a} for the bound of $\|\nabla^2 \tilde g\|_{L^2(\bm\pi(\Gamma_{h, j}', \delta))}$. Since
\begin{align*}
	\|\nabla e\|_{L^2(\Omega_h)} &\le \|\nabla(\tilde g - g_h)\|_{L^2(\Omega_{h, j}')} + C d_j^{-1} \|\tilde g - g_h\|_{L^2(\Omega_{h, j}')}, \notag \\
	\|\nabla(e - \mathring I_h e)\|_{L^2(\Omega_h)} &\le C(hd_j^{-1} \|\nabla(\tilde g - g_h)\|_{L^2(\Omega_{h, j}')} + d_j^{-1} \|\tilde g - g_h\|_{L^2(\Omega_{h, j}')}) + C_j h^2 d_j^{-N/2 - 1}
\end{align*}
(recall \eqref{eq: e1 - Ih e1}, \eqref{eq: e2 - Ih e2}, and $\mathring I_h e_2 = I_h e_2$), we see that the last term $(\mathring I_h e, - \Delta \tilde g)_{\Omega_h \setminus \Omega}$ is bounded by
\begin{equation} \label{eq3: bound of (11)}
	C \delta^{3/2} d_j^{-(N+1)/2} (\|\nabla(\tilde g - g_h)\|_{L^2(\Omega_{h, j}')} + d_j^{-1} \|\tilde g - g_h\|_{L^2(\Omega_{h, j}')} + C_j h^2 d_j^{-N/2 - 1}).
\end{equation}

Now, combining the estimates \RED{\eqref{eq1: bound of (11)}, \eqref{eq2: bound of (11)}, and \eqref{eq3: bound of (11)}} with \eqref{eq: start of local H1}, we deduce that
\begin{align*}
	\|\nabla(\tilde g - g_h)\|_{L^2(\Omega_{h, j})}^2
		&\le Ch d_j^{-1} \|\nabla(\tilde g - g_h)\|_{L^2(\Omega_{h, j}')}^2 + C d_j^{-1} \|\tilde g - g_h\|_{L^2(\Omega_{h, j}')}\|\nabla(\tilde g - g_h)\|_{L^2(\Omega_{h, j}')} \\
		&\qquad + C_j h^2 d_j^{-N/2 - 1} \|\nabla(\tilde g - g_h)\|_{L^2(\Omega_{h, j}')} \\
		&\qquad + C \delta^{3/2} d_j^{-(N+1)/2} (\|\nabla(\tilde g - g_h)\|_{L^2(\Omega_{h, j}')} + d_j^{-1} \|\tilde g - g_h\|_{L^2(\Omega_{h, j}')} + C_j h^2 d_j^{-N/2 - 1}).
\end{align*}
Taking the square root and multiplying by $d_j^{N/2}$ yield (note that $hd_j^{-1} \le L^{-1}$)
\begin{align}
	d_j^{N/2} \|\nabla(\tilde g - g_h)\|_{L^2(\Omega_{h, j})}
	&\le C L^{-1/2} d_j^{N/2} \|\nabla(\tilde g - g_h)\|_{L^2(\Omega_{h, j}')} + C_j^{1/2} (h^2d_j^{-1})^{1/2} (d_j^{N/2} \|\nabla(\tilde g - g_h)\|_{L^2(\Omega_{h, j}')})^{1/2} \notag \\
		&\qquad + C \Big( d_j^{N/2-1} \|\tilde g - g_h\|_{L^2(\Omega_{h, j}')} \Big)^{1/2} \Big( d_j^{N/2} \|\nabla(\tilde g - g_h)\|_{L^2(\Omega_{h, j}')} \Big)^{1/2} \notag \\
		&\qquad + C \delta^{3/4}d_j^{-1/4} \Big[ \Big( d_j^{N/2} \|\nabla(\tilde g - g_h)\|_{L^2(\Omega_{h, j}')} \Big)^{1/2} + \Big( d_j^{N/2-1} \|\tilde g - g_h\|_{L^2(\Omega_{h, j}')} \Big)^{1/2} \Big] \notag \\
		&\qquad + C_j h \delta^{3/4}d_j^{-3/4}. \label{eq: k=2 is essential}
\end{align}
By summation for $j = 0, \dots, J$ and $\sum_{j=0}^J d_j^{N/2} \|\cdot\|_{L^2(\Omega_{h, j}')} \le C \sum_{j=0}^J d_j^{N/2} \|\cdot\|_{L^2(\Omega_{h, j})}$, combined with the use of Cauchy--Schwarz inequalities, we have
\begin{align*}
	\sum_{j=0}^J d_j^{N/2} \|\nabla(\tilde g - g_h)\|_{L^2(\Omega_{h, j})}
	&\le \Big( C' L^{-1/2} + \frac14 \Big) \sum_{j=0}^J d_j^{N/2} \|\nabla(\tilde g - g_h)\|_{L^2(\Omega_{h, j})} + C L^{N/2+1} h^2d_0^{-1} \\
		&\qquad + C \sum_{j=0}^J d_j^{N/2-1} \|\tilde g - g_h\|_{L^2(\Omega_{h, j})} + C \delta^{3/2}d_0^{-1/2} + CL^{N/2+1} h \delta^{3/4}d_0^{-3/4},
\end{align*}
which concludes, if $L$ is chosen large enough to satisfy $C' L^{-1/2} \le 1/4$,
\begin{equation*}
	\sum_{j=0}^J d_j^{N/2} \|\nabla(\tilde g - g_h)\|_{L^2(\Omega_{h, j})} \le CL^{N/2+1} h + C \sum_{j=0}^J d_j^{N/2-1} \|\tilde g - g_h\|_{L^2(\Omega_{h, j})}.
\end{equation*}
Therefore, \eqref{eq: weighted H1} is proved, once a weighted $L^2$-error estimate of the form
\begin{equation} \label{eq: weighted L2}
	\sum_{j=0}^J d_j^{N/2-1} \|\tilde g - g_h\|_{L^2(\Omega_{h, j})} \le C L^{-1} \sum_{j=0}^J d_j^{N/2} \|\nabla(\tilde g - g_h)\|_{L^2(\Omega_{h, j})} + Ch
\end{equation}
is established ($L$ is chosen sufficiently large again).

\begin{remark}
	The assumption $k \ge 2$, that is, we restrict our consideration to quadratic or higher-order elements, is essential because for $k = 1$ (linear element) we would have only a factor $h$ instead of $h^2 d_j^{-1}$ in \eqref{eq: k=2 is essential}, which results in an $O(h|\log h|)$-term after the summation over $j$.
\end{remark}

\section{Weighted $L^2$-estimate} \label{sec6}
Let us prove \eqref{eq: weighted L2} above.
First we fix $j = 0, \dots, J$ and estimate $\|\tilde g - g_h\|_{L^2(\Omega_{h, j})}$.
To this end, for arbitrary $\psi \in C^\infty_0(\Omega_{h, j})$ such that $\|\psi\|_{L^2(\Omega_{h, j})} = 1$ we define $w \in H^2(\Omega)$ to be the solution of a dual problem ($\psi$ is extended by zero outside $\Omega_{h, j}$)
\begin{equation*}
	-\Delta w = \psi \quad\text{in }\; \Omega, \qquad w = 0 \quad\text{on }\; \Gamma.
\end{equation*}

Setting $\tilde w := Pw$, we observe (recall $a_h(\cdot, \tilde g - g_h) = (\cdot, -\Delta \tilde g)_{\Omega_h \setminus \Omega}$ in $\mathring V_h$):
\begin{equation} \label{eq: start of local L2}
\begin{aligned}
	(\psi, \tilde g - g_h)_{\Omega_h} &= (-\Delta w, \tilde g - g_h)_{\Omega_h \cap \Omega} + (\psi, \tilde g - g_h)_{\Omega_h \setminus \Omega} \\
		&= (-\Delta w, \tilde g - g_h)_{\Omega_h} + (\Delta \tilde w + \psi, \tilde g - g_h)_{\Omega_h \setminus \Omega} \\
		&= a_h(\tilde w, \tilde g - g_h) - (\partial_{\bm n_h} \tilde w, \tilde g - g_h)_{\Gamma_h} + (\Delta \tilde w + \psi, \tilde g - g_h)_{\Omega_h \setminus \Omega} \\
		&= a_h(\tilde w - \mathring I_h \tilde w, \tilde g - g_h) - (\mathring I_h \tilde w, \Delta \tilde g)_{\Omega_h \setminus \Omega}
			- (\partial_{\bm n_h} \tilde w, \tilde g)_{\Gamma_h} + (\Delta \tilde w + \psi, \tilde g - g_h)_{\Omega_h \setminus \Omega}.
\end{aligned}
\end{equation}

The first term on the right-hand side of \eqref{eq: start of local L2} is bounded, using Proposition \ref{prop: I - mathring Ih local}, by
\begin{align*}
	&\sum_{\ell = 0}^J \|\nabla(\tilde w - \mathring I_h \tilde w)\|_{L^2(\Omega_{h, \ell})} \|\nabla(\tilde g - g_h)\|_{L^2(\Omega_{h, \ell})}
		\le Ch \sum_{\ell = 0}^J \|\tilde w\|_{H^2(\Omega_{h, \ell} \cup \bm\pi(\Gamma_{h, \ell}, \delta))} \|\nabla(\tilde g - g_h)\|_{L^2(\Omega_{h, \ell})}.
\end{align*}
If $\ell \le j - 2$, since $\operatorname{dist}(\Omega_{h, \ell}, \Omega_{h, j}) \ge \frac14 d_j$ and $d_j \ge d_\ell$ we obtain
\begin{align*}
	\|\tilde w\|_{H^2(\Omega_{h, \ell} \cup \bm\pi(\Gamma_{h, \ell}, \delta))} \le C d_\ell^{N/2} \cdot d_j^{N/2} d_j^{-N} = C d_\ell^{N/2} d_j^{-N/2},
\end{align*}
where we have used the representation formula $w(\bm x) = \int_{\Omega \cap \Omega_{h, j}} G(\bm x, \bm y) \psi(\bm y) \, d\bm y$ as well as derivative estimates for $G(\bm x, \bm y)$, cf.\ \cite[Lemma B.4]{KasKem2020a}.
Similarly, if $\ell \ge j + 2$, this quantity is bounded by $C d_\ell^{-N/2} d_j^{N/2} \le C d_\ell^{N/2} d_j^{-N/2}$ using the facts $\operatorname{dist}(\Omega_{h, \ell}, \Omega_{h, j}) \ge \frac12 d_\ell$ and $d_\ell \ge d_j$.
For the remaining case $j-1 \le \ell \le j+1$, it follows from the global elliptic $H^2$-regularity estimate that
\begin{align*}
	\|\tilde w\|_{H^2(\Omega_{h, \ell} \cup \bm\pi(\Gamma_{h, \ell}, \delta))} \le C \|w\|_{H^2(\Omega)} \le C\|\psi\|_{L^2(\Omega)} \le C.
\end{align*}
Consequently,
\begin{equation} \label{eq1: bound of (psi, g - gh)}
	|a_h(\tilde w - \mathring I_h \tilde w, \tilde g - g_h)| \le C hd_j^{-N/2} \sum_{\ell = 0}^J d_\ell^{N/2} \|\nabla(\tilde g - g_h)\|_{L^2(\Omega_{h, \ell})}.
\end{equation}

For the second term of \eqref{eq: start of local L2}, we see that
\begin{align*}
	|(\mathring I_h \tilde w, \Delta \tilde g)_{\Omega_h \setminus \Omega}| &\le \sum_{\ell=0}^J \|\mathring I_h \tilde w\|_{L^2(\Omega_{h, \ell} \setminus \Omega)} \|\nabla^2 \tilde g\|_{L^2(\Omega_{h, \ell} \setminus \Omega)},
\end{align*}
where $\|\nabla^2 \tilde g\|_{L^2(\Omega_{h, \ell} \setminus \Omega)}$ can be bounded by $C \delta^{1/2} d_\ell^{-(N+1)/2}$, cf.\ \cite[Lemma B.3]{KasKem2020a}.
Moreover, by \eqref{eq: RHS with Omegah minus Omega} (note that $\mathring I_h \tilde w = 0$ on $\Gamma_h$), \eqref{eq2: boundary-skin estimates}, and \pref{prop: I - mathring Ih local} (with $v_2 = 0$),
\begin{align*}
	\|\mathring I_h \tilde w\|_{L^2(\Omega_{h, \ell} \setminus \Omega)} &\le C \delta \|\nabla (\mathring I_h \tilde w)\|_{L^2(\Omega_{h, \ell} \setminus \Omega)} \\
		&\le C \delta (\|\nabla\tilde w\|_{L^2(\Omega_{h, \ell} \setminus \Omega)} + \|\nabla (\mathring I_h \tilde w - \tilde w)\|_{L^2(\Omega_{h, \ell})}) \\
		&\le C \delta^{3/2} \|\nabla\tilde w\|_{L^2(\Gamma_{h, \ell})} + C \delta h\|\tilde w\|_{H^2(\Omega_{h, \ell} \cup \bm\pi(\Gamma_{h, \ell}, \delta))} \\
		&\le \begin{cases}
			C \delta^{3/2} d_\ell^{(N-1)/2} d_j^{-N/2+1} + Ch\delta d_\ell^{N/2} d_j^{-N/2} & (\ell \le j - 2), \\
			C \delta^{3/2} + Ch\delta & (\ell = j, j\pm1), \\
			C \delta^{3/2} d_\ell^{(1-N)/2} d_j^{N/2} + Ch\delta d_\ell^{-N/2} d_j^{N/2} & (\ell \ge j + 2),
		\end{cases}
\end{align*}
where we have used (cf.\ \cite[Lemma B.4]{KasKem2020a})
\begin{equation} \label{eq: nabla w on Gamma hl}
	\|\nabla\tilde w\|_{L^2(\Gamma_{h, \ell})} \le
	\begin{cases}
		C\|\nabla \tilde w\|_{H^1(\Omega_h)} \le C &\quad (\ell =  j, j\pm1), \\
		Cd_\ell^{(N-1)/2} \cdot d_j^{N/2} \max\{d_\ell, d_j\}^{1-N} &\quad (\ell \neq j, j\pm1).
	\end{cases}
\end{equation}
Consequently, it follows that
\begin{equation} \label{eq2: bound of (psi, g - gh)}
\begin{aligned}
	|(\mathring I_h \tilde w, \Delta \tilde g)_{\Omega_h \setminus \Omega}| &\le C \sum_{\ell=0}^{j-2} (\delta^2 d_\ell^{-1} d_j^{-N/2+1} + h\delta^{3/2} d_\ell^{-1/2} d_j^{-N/2}) + Ch\delta^{3/2} d_j^{-(N+1)/2} \\
		&\qquad + C \sum_{\ell=j+2}^{J} (\delta^2 d_\ell^{-N} d_j^{N/2} + h\delta^{3/2} d_\ell^{-N-1/2} d_j^{N/2}) \\
		&\le C \delta^2 d_j^{-N/2} + Ch \delta^{3/2} d_j^{-(N+1)/2}.
\end{aligned}
\end{equation}

The third term of \eqref{eq: start of local L2} is bounded by (note that $g \circ \bm\pi = 0$ on $\Gamma_h$)
\begin{equation} \label{eq3: bound of (psi, g - gh)}
\begin{aligned}
	&\sum_{\ell = 0}^J \|\nabla\tilde w\|_{L^2(\Gamma_{h, \ell})} \|\tilde g - g\circ\bm\pi\|_{L^2(\Gamma_{h, \ell})} \\
	\le \; &\sum_{\ell = 0}^J \|\nabla\tilde w\|_{L^2(\Gamma_{h, \ell})} (C \delta \|\nabla \tilde g\|_{L^2(\Gamma_{h, \ell})} + C \delta^{3/2} \|\nabla^2 \tilde g\|_{L^2(\bm\pi(\Gamma_{h, \ell}, \delta))}) \\
	\le \; & C \Big( \sum_{\ell = 0}^{j-2} d_\ell^{(N-1)/2} d_j^{-N/2+1} \cdot \delta d_\ell^{(1-N)/2} + \delta d_j^{(1-N)/2}
		+ \sum_{\ell = j+2}^{J} d_\ell^{(1-N)/2} d_j^{N/2} \cdot \delta d_\ell^{(1-N)/2} \Big) \\
	\le \; &C|\log h| \delta d_j^{-N/2+1} + C \delta d_j^{(1-N)/2},
\end{aligned}
\end{equation}
where we have used \eqref{eq3: boundary-skin estimates} and \eqref{eq2: boundary-skin estimates} in the second line, and \eqref{eq: nabla w on Gamma hl} and \cite[Lemma B.3]{KasKem2020a} in the third line.

Finally, recalling the argument before \eqref{eq1: bound of (psi, g - gh)}, especially $d_\ell \lesseqgtr d_j$ for $\ell \lesseqgtr j$, we get
\begin{align*}
	\|\Delta \tilde w + \psi\|_{L^2(\Omega_{h, \ell})} \le
	\begin{cases}
		C d_\ell^{N/2} d_j^{-N/2} \quad &(\ell \le j), \\
		C d_\ell^{-N/2} d_j^{N/2} \quad &(\ell \ge j+1).
	\end{cases}
\end{align*}
In addition, since $g_h = g \circ \bm\pi = 0$ on $\Gamma_h$,
\begin{align}
	\|\tilde g - g_h\|_{L^2(\Omega_{h, \ell} \setminus \Omega)} &\le C\delta^{1/2} \|\tilde g\|_{L^2(\Gamma_{h, \ell})} + C \delta \|\nabla(\tilde g - g_h)\|_{L^2(\Omega_{h, \ell} \setminus \Omega)} \tag{by \eqref{eq: RHS with Omegah minus Omega}} \\
		&\le C \delta (\|\nabla(\tilde g - g_h)\|_{L^2(\Omega_{h, \ell})} + \|\nabla \tilde g\|_{L^2(\bm\pi(\Gamma_{h, \ell}, \delta))}) \tag{by \eqref{eq3: boundary-skin estimates}} \\
		&\le C \delta \|\nabla(\tilde g - g_h)\|_{L^2(\Omega_{h, \ell})} + C \delta^{3/2} d_\ell^{(1-N)/2}. \notag
\end{align}
Therefore, we can bound the last term $(\Delta \tilde w + \psi, \tilde g - g_h)_{\Omega_h \setminus \Omega}$ of \eqref{eq: start of local L2} by
\begin{align}
	&C\sum_{\ell=0}^j d_\ell^{N/2} d_j^{-N/2} (\delta \|\nabla(\tilde g - g_h)\|_{L^2(\Omega_{h, \ell})} + \delta^{3/2} d_\ell^{(1-N)/2}) \notag \\
	+ \; &C\sum_{\ell=j+1}^J d_\ell^{-N/2} d_j^{N/2} (\delta \|\nabla(\tilde g - g_h)\|_{L^2(\Omega_{h, \ell})} + \delta^{3/2} d_\ell^{(1-N)/2}) \notag \\
	\le \; &C \delta d_j^{-N/2} \sum_{\ell=0}^J d_\ell^{N/2} \|\nabla(\tilde g - g_h)\|_{L^2(\Omega_{h, \ell})} + C \delta^{3/2} d_j^{(1-N)/2}. \label{eq4: bound of (psi, g - gh)}
\end{align}

Combining the estimates \eqref{eq1: bound of (psi, g - gh)}, \eqref{eq2: bound of (psi, g - gh)}, \eqref{eq3: bound of (psi, g - gh)}, \eqref{eq4: bound of (psi, g - gh)} with \eqref{eq: start of local L2} and noting that $\|\tilde g - g_h\|_{L^2(\Omega_{h, j})}$ equals the supremum of $(\psi, \tilde g - g_h)_{\Omega_h}$ with respect to $\psi \in C^\infty_0(\Omega_{h, j})$ such that $\|\psi\|_{L^2(\Omega_{h, j})} = 1$, we deduce that
\begin{align*}
	d_j^{N/2-1} \|\tilde g - g_h\|_{L^2(\Omega_{h, j})}
		&\le C hd_j^{-1} \sum_{\ell = 0}^J d_\ell^{N/2} \|\nabla(\tilde g - g_h)\|_{L^2(\Omega_{h, \ell})} + C (\delta^2 d_j^{-1} + h \delta^{3/2} d_j^{-3/2}) \\
		&\quad + C\delta (|\log h| + d_j^{-1/2}) + C \delta d_j^{-1} \sum_{\ell=0}^J d_\ell^{N/2} \|\nabla(\tilde g - g_h)\|_{L^2(\Omega_{h, \ell})} + C \delta^{3/2} d_j^{-1/2}.
\end{align*}
We add this for $j = 0, \dots, J$ and notice $hd_0 \le L^{-1}$ to conclude
\begin{align*}
	\sum_{j=0}^J d_j^{N/2-1} \|\tilde g - g_h\|_{L^2(\Omega_{h, j})} &\le CL^{-1} \sum_{\ell = 0}^J d_\ell^{N/2} \|\nabla(\tilde g - g_h)\|_{L^2(\Omega_{h, \ell})}
		+ C (\delta^2 h^{-1} + \delta^{3/2} h^{-1/2} + \delta |\log h|^2 + \delta h^{-1/2}).
\end{align*}
Since $\delta = Ch^{k+1}$ (in particular less than $Ch^2)$, we have shown \eqref{eq: weighted L2}, thus completing the proof of Theorem \ref{main thm}.

\section{Appendix}
\subsection{Proof of Proposition \ref{prop: regularized delta} (cf.\ \cite[Appendix]{SSW96})}
Fix a non-negative cut off function $\hat \omega \in C^\infty_0(\hat T)$ in the reference element $\hat T$ such that $\operatorname{supp} \hat\omega$ contains an open ball and is contained in $\operatorname{int} \hat T$.
It follows that $\operatorname{dist}(\operatorname{supp} \hat\omega, \partial\hat T) \ge C > 0$ and $\|\nabla^m \hat\omega\|_{\RED{L^\infty(\hat T)}} \le C \, (m \ge 0)$.
We set $\omega_K := \hat\omega \circ \bm F_K^{-1}$.

Noticing that $\{\phi_{\bm a}\}_{\bm a \in \mathcal N_h \cap K}$ forms a basis of $V_h|_K$, we construct its dual basis $\{\psi_{\bm b}\}_{\bm b \in \mathcal N_h \cap K}$ with respect to the inner product $(\cdot, \omega_K \cdot)_{K}$, that is,
\begin{equation*}
	\int_K \phi_{\bm a}(\bm x) \omega_K(\RED{\bm x}) \psi_{\bm b}(\bm x) \, d\bm x = \int_{\hat T} \hat\phi_i(\hat{\bm x}) \hat\omega(\hat{\bm x}) \hat\psi_j(\hat{\bm x}) J_K \, d\hat{\bm x} = \delta_{ij},
\end{equation*}
where the indices $i$ and $j$ satisfy $\bm F_T(\hat{\bm a}_i) = \bm a$, $\bm F_T(\hat{\bm a}_j) = \bm b$ and $J_K := |\operatorname{det} \nabla_{\hat{\bm x}}\bm F_K|$.

Then we define $\eta_{K, \bm z}$ by
\begin{equation*}
	\eta_{K, \bm z}(\bm x) = \sum_{\bm a \in \mathcal N_h \cap K} \phi_{\bm a}(\bm z) \omega_K(\bm x) \psi_{\bm a}(\bm x) \quad (\bm x \in K).
\end{equation*}
By direct computation one can verify $(\eta_{K, \bm z}, v_h)_K = \RED{v_h(\bm z)}$ for any $v_h \in V_h|_K$, and $\operatorname{dist}(\operatorname{supp} \omega_K, \partial K) \ge Ch_K$ (recall $\|\nabla_{\bm x} \bm F_K^{-1}\|_{L^\infty(K)} \le Ch_K^{-1}$).

Finally observe that $\psi_{\bm a} = \sum_{j=1}^{N_k} (M^{-1})_{ij} \phi_{\bm b}$ where $\bm a = \bm F_K(\hat {\bm a}_i)$, $\bm b = \bm F_K(\hat {\bm a}_j)$, and $M^{-1}$ denotes the inverse matrix of $M = ((\hat\phi_j, \hat\omega J_K \hat\phi_l)_K)_{1 \le j, l \le N_k}$.
Then the derivative estimates for $\eta_{K, \bm z}$ result from $|J_K| \approx h_K^N$ and $\|\nabla^m (\omega \phi_{\bm b})\|_{L^\infty(K)} \le Ch_K^{-m}$.
This completes the proof of Proposition \ref{prop: regularized delta}.

\subsection{Interpolation error estimates for $\mathring I_h$}
\begin{proof}[Proof of Lemma \ref{lem: Ih - mathring Ih}]
	In view of the inverse inequality $\|\nabla^m (v_h - \mathring I_h v_h)\|_{L^p(T)} \le Ch_T^{-m} \|v_h - \mathring I_h v_h\|_{L^p(T)}$, it is enough to consider $m = 0$.
	We see that, for $\bm x \in T$,
	\begin{equation*}
		|(v_h - \mathring I_h v_h)(\bm x)|^p = \Big| \sum_{\bm a \in \mathcal N_h^\partial \cap T} v_h(\bm a) \phi_{\bm a}(\bm x) \Big|^p \le C \|v_h\|_{L^\infty(S)}^p \sum_{\bm a \in \mathcal N_h \cap T} |\phi_{\bm a}(\bm x)|^p,
	\end{equation*}
	where we note that $\mathcal N_h^\partial \cap T \subset S$ by assumption.
	Then it follows that
	\begin{align*}
		\int_T \sum_{\bm a \in \mathcal N_h \cap T} |\phi_{\bm a}(\bm x)|^p \, d\bm x = \sum_{i=1}^{N_k} \int_{\hat T} |\hat\phi_i(\hat{\bm x})|^p |\operatorname{det} \nabla_{\hat{\bm x}}\bm F_T| \, d\hat{\bm x} \le Ch_T^N.
	\end{align*}
	Next let $T' \in \mathcal T_h$ denote the element that contains $S$.
	We may assume that $\hat S = \bm F_{T'}^{-1}(S)$, up to a further orthogonal transformation, is contained in the plane $\hat x_N = 0$.
	The Riemannian metric tensor $G$ obtained from the parametrization $(\hat{\bm x}', 0) \mapsto \bm F_{T'}(\hat{\bm x}', 0)$ satisfies $C h_{T'}^{N-1} \le \sqrt{\operatorname{det} G} \le C h_{T'}^{N-1}$.
	Then we have
	\begin{align*}
		\|v_h\|_{L^\infty(S)}^p = \|\hat v_h\|_{L^\infty(\hat S)}^p \le C \|\hat v_h\|_{L^p(\hat S)}^p &\le C h_T^{1-N} \int_{\hat S} |\hat v_h(\hat{\bm x}', 0)|^p \sqrt{\operatorname{det} G} \, d\hat{\bm x}'
			= C h_T^{1-N} \|v_h\|_{L^p(S)}^p,
	\end{align*}
	where we have used norm equivalence for $\mathbb P_k(\hat S)$ and $h_T \approx h_{T'}$ (by the shape-regularity).
	From these we conclude the desired estimate for $m = 0$.
\end{proof}

\begin{proof}[Proof of Proposition \ref{prop: I - mathring Ih local}]
	If $T \in \mathcal T_h, \, T \subset \Omega_{h, j}$ has no intersection with $\Gamma_h$, we have $(\mathring I_h v)|_T = (I_h v)|_T$ so that estimation of $\|\nabla^l (v - \mathring I_h v)\|_{L^p(T)}$ is reduced to Lemma \ref{lem: I - Ih}.
	
	If $T \cap \Gamma_h \neq \emptyset$ one can find $S_T \in \mathcal S_h$ such that $T \cap \Gamma_h \subset S_T \subset \Gamma_{h, j}$ (by the shape-regularity, $\# \{T \in \mathcal T_h \mid S_T = S\}$ is bounded independently of $h$ for all $S \in \mathcal S_h$).
	By Lemmas \ref{lem: Ih - mathring Ih} and \ref{lem: I - Ih} we obtain
	\begin{align*}
		\|\nabla^l (v - \mathring I_h v)\|_{L^p(T)} &\le \|\nabla^l (v - I_h v)\|_{L^p(T)} + \|\nabla^l (I_h v - \mathring I_h v)\|_{L^p(T)} \\
			&\le Ch_T^{m-l} \|v\|_{W^{m,p}(T)} + Ch_T^{1/p - l} \|I_h v\|_{L^p(S_T)} \\
			&\le Ch_T^{m-l} \|v\|_{W^{m,p}(T)} + Ch_T^{1/p - l} (\|I_h v - v\|_{L^p(S_T)} + \|v\|_{L^p(S_T)}).
	\end{align*}
	Let $T' \in \mathcal T_h$ be the element that contains $S_T$ (hence $T' \subset \Omega_{h, j}$).
	Applying the following trace inequality between $T'$ and $S_T$:
	\begin{equation*}
		\|f\|_{L^p(S_T)} \le C h_T^{-1/p} (\|f\|_{L^p(T')} + h_T \|f\|_{W^{1,p}(T')}) \quad (f \in W^{1,p}(T')),
	\end{equation*}
	we have $\|I_h v - v\|_{L^p(S_T)} \le C h_T^{m-1/p} \|v\|_{W^{m,p}(T')}$.
	
	It follows from $v_1 \circ \bm\pi = v_2 = 0$ on $\Gamma_h$ and $\delta^{1-1/p} \le 1$ that
	\begin{align*}
		\|v\|_{L^p(S_T)} &= \|v_1 - v_1 \circ \bm\pi\|_{L^p(S_T)} \le C \delta^{1-1/p} \|\nabla v_1\|_{L^p(\bm\pi(S_T, \delta))} \\
			&\le C \delta \|\nabla v_1\|_{L^p(S_T)} + C \delta^{2-1/p} \|\nabla^2 v_1\|_{L^p(\bm\pi(S_T, \delta))} \\
			&\le Ch_T^{-1/p}\delta (\|v_1\|_{W^{1,p}(T')} + h_T \|v_1\|_{W^{2,p}(T')}) + C \delta \|\nabla^2 v_1\|_{L^p(\bm\pi(S_T, \delta))}.
	\end{align*}
	Combination of the above estimates, together with $\delta \le Ch_T^{k+1}$, leads to
	\begin{align*}
		\|\nabla^l (v - \mathring I_h v)\|_{L^p(T)} &\le Ch^{m-l} (\|v\|_{W^{m,p}(T)} + \|v\|_{W^{m,p}(T')}) \\
			&\qquad + C h^{k+1-l} (\|v_1\|_{W^{2,p}(T')} + \|\nabla^2 v_1\|_{L^p(\bm\pi(S_T, \delta))}).
	\end{align*}
	Raising both sides to the power $p$ and adding for $T \subset \Omega_{h, j}$, we conclude \eqref{eq: I - mathring Ih local}.
\end{proof}

\section*{Acknowledgments}
This is a manuscript submitted for the Proceedings of MSJ-KMS Joint Meeting 2023.
The author thanks the anonymous referee for valuable comments on this paper.
This work was supported by Grant-in-Aid for Early-Career Scientists (No.\ 20K14357) and Grant-in-Aid for Scientific Research(C) (No.\ 24K06860) of the Japan Society for the Promotion of Science (JSPS).




\end{document}